\documentclass[11pt]{article}
\usepackage{amsmath,amssymb,amsthm,eucal, mathrsfs, mathtools,hyperref}
\usepackage{xcolor}
\usepackage[all]{xy}
\usepackage{wasysym}

\title{The structure of the wave operator in four dimensions in the presence of resonances}
\author{Angus Alexander, 
Adam Rennie\thanks{email: 
\texttt{angusa@uow.edu.au, renniea@uow.edu.au}}
\\[3pt]
School of Mathematics and 
Applied Statistics, University of Wollongong,\\
Wollongong, Australia\\
}

\advance\topmargin by -\headheight
\advance\topmargin by -\headsep
\evensidemargin=\oddsidemargin
\makeatletter
\def\section{\@startsection{section}{1}{\z@}{-3.5ex plus -1ex minus
  -.2ex}{2.3ex plus .2ex}{\large\bf}}
\def\subsection{\@startsection{subsection}{2}{\z@}{-3.25ex plus -1ex
  minus -.2ex}{1.5ex plus .2ex}{\normalsize\bf}}
\makeatother

\numberwithin{equation}{section} 

\theoremstyle{plain} 
\newtheorem{thm}{Theorem}[section]
\newtheorem{lemma}[thm]{Lemma}
\newtheorem{prop}[thm]{Proposition}
\newtheorem{cor}[thm]{Corollary}

\theoremstyle{definition} 
\newtheorem{defn}[thm]{Definition}

\theoremstyle{remark} 

\DeclareMathOperator{\Dom}{Dom}   

\newcommand{\eps}{\varepsilon} 

\newcommand{\B}{\mathcal{B}}  
\newcommand{\C}{\mathbb{C}}   
\renewcommand{\d}{\mathrm{d}} 

\newcommand{\e}{\mathrm{e}}
\newcommand{\F}{\mathcal{F}}  
\newcommand{\Fi}{\mathcal{F}} 
\renewcommand{\H}{\mathcal{H}}  
\newcommand{\K}{\mathcal{K}}  
\newcommand{\norm}[1]{\left|\left|#1\right|\right|} 
\newcommand{\ox}{\otimes}     
\newcommand{\R}{\mathbb{R}}   
\newcommand{\Sf}{\mathbb{S}}  

\newcommand{\vp}{\varphi}
\newcommand{\Z}{\mathbb{Z}} 

\newcommand{\stroke}{\mathbin|}     
\newcommand{\od}[2]{\frac{\mathrm{d}#1}{\mathrm{d}#2}}

\def\pairL_#1(#2|#3){{}_{#1}(#2\stroke#3)} 
\def\pairR(#1|#2)_#3{(#1\stroke#2)_{#3}} 
\def\scal<#1|#2>{\langle#1\stroke#2\rangle} 

\renewcommand{\epsilon}{\varepsilon}

\theoremstyle{definition}

\definecolor{MyBlue}{cmyk}{1,0.13,0,0.63}
\definecolor{MyGreen}{cmyk}{0.91,0,0.88,0.52}
\newcommand{\mylinkcolor}{MyBlue}
\newcommand{\mycitecolor}{MyGreen}
\newcommand{\myurlcolor}{black}

\usepackage{hyperref}
\hypersetup{%
  bookmarksnumbered=true,bookmarksopen=false,%
  plainpages=false,
  linktocpage=true,%
  colorlinks=true,breaklinks=true,%
  linkcolor=\mylinkcolor,citecolor=\mycitecolor,urlcolor=\myurlcolor,%
  pdfpagelayout=OneColumn,%
  pageanchor=true,%
}

\begin{document}

\maketitle

\vspace{-2pc}

\begin{abstract}
We show that the wave operators for Schr\"{o}dinger scattering in $\R^4$ have a particular form which depends on the existence of resonances. As a consequence of this form, we determine the contribution of resonances to the index of the wave operator.
\end{abstract}
\maketitle

\parindent=0.0in
\parskip=0.00in


\parskip=0.06in

\section{Introduction}
\label{sec:intro}

In this paper we study scattering for Schr\"{o}dinger operators in $\R^4$. In particular we analyse the structure of the wave operator in terms of the scattering operator, the generator of dilations, and a novel contribution corresponding to the existence of resonances. The form of the wave operator allows us to deduce the contribution of resonances to Levinson's theorem via the index of the wave operator.

Levinson's theorem \cite{levinson49} gives in dimension $n = 1$ that the number of eigenvalues (counted with multiplicity) of the Schr\"{o}dinger operator $H= H_0+V$ (with $H_0 = -\Delta$) for a suitably decaying potential $V$ satisfies
\begin{align}
N &= \frac{1}{\pi} \left(\delta(0)-\delta(\infty) \right) +\frac12 \nu,
\label{eq:Lev}
\end{align}
where $\delta$ is the scattering phase and $\nu \in \{0,1\}$ depends on the existence of a resonance (a distributional solution to $(-\Delta+V)\psi = 0$ with $\psi \notin L^2(\R)$).

In dimension $n = 3$, it was shown in \cite[Section 5]{newton60} that for spherically symmetric potentials, Equation \eqref{eq:Lev} holds for each angular momentum mode. This result was extended by \cite[Section 7]{newton77} for general potentials, and it was shown that resonances can occur and provide a half-integer contribution to a Levinson-type theorem.

One approach to resonances is the threshold behaviour of the spectrum of Schr\"{o}dinger operators as the strength of the potential is scaled, $\lambda\to \lambda V$. A resonance occurs just as an eigenvalue emerges from the continuous spectrum. Klaus and Simon \cite[Table I]{KS80} determined in dimension $n = 2, 4$ that the logarithmic behaviour of the singularity of the free resolvent $R_0(z) = (H_0-z)^{-1}$ near $z = 0$ gives resonances with different properties to those found in odd dimensions. In particular, in the case of a spherically symmetric potential the `resonances' in dimension $n = 2$ are either $s$-resonances or $p$-resonances (corresponding to angular momentum $\ell = 0$ or $\ell = 1$ respectively) whilst the resonances in dimension $n = 4$ are $s$-resonances. It was observed in \cite{jensen80, KS80} that no such phenomena can occur in dimension $n \geq 5$. 

In \cite{jensen84} a precise low energy expansion of the resolvents $R_0(z)$ and $R(z) = (H-z)^{-1}$ was given in dimension $n =4$, providing the definition of a resonance as a solution $\psi$ to the equation $H\psi = 0$ which is not square-integrable but lies in some weighted Sobolev space. Obtaining such an expansion in dimension $n = 2$ was a more difficult task. In \cite{bolle86, bolle88} a low energy expansion of the resolvent $R(z)$ was provided in dimension $n = 2$ to give a precise definition of resonances and determine their contribution to Levinson's theorem. It was shown that $s$-resonances give no contribution to Levinson's theorem and $p$-resonances give an integer contribution, both contrasting the behaviour of resonances in dimension $n = 1,3$.

More recently in dimension $n = 4$ it was shown in \cite[Equation (26b)]{dong02} that in the case of a sufficiently decaying spherically symmetric potential, the $s$-resonances provide an integer contribution to Levinson's theorem. In \cite[Theorem 1.1]{jia12} this result was generalised to non-spherically symmetric potentials.

In \cite{jensen01} a new symmetrised technique was introduced for performing low-energy expansions of an operator related to $R(z)$ in terms of powers of a single variable, avoiding the double Laurent expansion of \cite{bolle88}. The technique in \cite{jensen01} allows us to systematically isolate the behaviour at zero in the spectrum of $H$ into the range of a decreasing sequence of finite rank projections. This technique is described in detail for four dimensional Schr\"{o}dinger operators in \cite{erdogan14}. See also  \cite{goldberg17, green17,yajima22}.

In \cite[Theorem 1.1]{richard13}, \cite[Theorem 1.3]{richard21} and \cite[Theorem 3.1]{AR23} it was shown that for all $n \geq 2$ and suitably decaying potentials $V$ (with the additional assumption that there are no $p$-resonances in dimension $n = 2$ and no resonances in dimension $n = 4$), the wave operator $W_-$ is of the form
\begin{align}\label{eq:wave-op}
W_- &= \textup{Id} + \varphi(D_n)(S-\textup{Id})+K,
\end{align}
where $D_n$ is the generator of the dilation group on $L^2(\R^n)$,  $\vp:\R \to \C$ is given by
\begin{align}\label{eq:vp-defn}
\vp(x) &= \frac12 \left( 1+\tanh{(\pi x)} -i \cosh{(\pi x)}^{-1} \right)
\end{align}
and $K$ is a compact operator. 

In recent years much work has been done on developing formulae analogous to Equation \eqref{eq:wave-op} for various scattering systems, including Schr\"{o}dinger scattering \cite{AR23, kellendonk08, kellendonk12, richard13, richard13ii, richard21}, point interactions \cite{kellendonk06}, rank-one perturbations \cite{richard10}, the Friedrichs-Faddeev model \cite{isozaki12},  Aharanov-Bohm operators \cite{kellendonk11}, lattice scattering \cite{bellissard12}, half-line scattering \cite{inoue20}, discrete scattering \cite{inoue19}, scattering for an inverse-square potential \cite{DR, inoue19ii} and for 1D Dirac operators with zero-range interections \cite{richard14}.

The exclusion of $p$-resonances in dimension $n = 2$ and resonances in dimension $n = 4$ to obtain Equation \eqref{eq:wave-op} is due to the presence of an additional logarithmic singularity near zero in the resolvent expansion provided by \cite[Proposition 5.3]{erdogan14}. Recent work \cite[Lemma 4.3]{ANRR} shows, using the resolvent expansions of \cite{jensen01}, that in the presence of $p$-resonances in dimension $n = 2$ the wave operator satisfies
\begin{align}\label{eq:wave-op-res}
W_- &= \textup{Id}+ \vp(D_n)(S-\textup{Id}) + B_{res} + K,
\end{align}
where $\varphi$ is given in Equation \eqref{eq:vp-defn}, $B_{res}$ is a bounded operator depending on the existence of $p$-resonances and $K$ is a compact operator. 

The nature of the singularity in the resolvent at zero energy in the presence of resonances in dimension $n = 4$ is the same as that of $p$-resonances in dimension $n = 2$. As such, we can use the techniques of \cite{ANRR} to show in Theorem \ref{thm:main-wave} that Equation \eqref{eq:wave-op-res} holds in dimension $n= 4$ with $B_{res}$ an operator depending on the existence of resonances. 


After factorising the wave operator as the product of two Fredholm operators, we show in Corollary \ref{cor:index-resonance} that
\begin{align}\label{eq:index}
\textup{Index}(W_-) &= \textup{Index}(W_S) + \dim(P_s),
\end{align}
where $\dim(P_s) \in \{0,1\}$ is the number of linearly independent $s$-resonances and $W_S = \textup{Id}+\vp(D_4)(S-\textup{Id})$ is a Fredholm operator.

For a comparison, we recall that in \cite[Theorem 1.1]{jia12} it is proved that for a sufficiently rapidly decaying potential the number of eigenvalues $N$ of $H$ (counted with multiplicity) is given by
\begin{align}\label{eq:jia-levinson}
-N &= \frac{1}{2\pi i} \int_0^\infty \left( \textup{Tr} \left(S(\lambda)^*S'(\lambda) \right) - c_1 \right) \, \d \lambda  - \beta_2 + \dim(P_s)
\end{align}
for some constants $c_1, \beta_2 \in \C$ depending on the potential. Even after taking into account the equality $\textup{Index}(W_-) = -N$, Equation \eqref{eq:index} cannot be immediately compared with Equation \eqref{eq:jia-levinson} other than to note that the existence of a resonance provides an integer contribution to both. To show that Equation \eqref{eq:jia-levinson} follows from Equation \eqref{eq:index} requires a more subtle analysis of the high-energy behaviour of the scattering matrix in the trace norm, for which we defer to a future publication.

The layout of the paper is as follows. In Section \ref{sec:scats} we introduce the relevant concepts from scattering theory for four-dimensional Schr\"{o}dinger operators and fix our notation. In particular, we discuss the low-energy resolvent expansion provided by \cite{erdogan14, jensen01}, the definition of a resonance and the properties of the scattering matrix. In Section \ref{sec:wave-op} we prove, via a number of technical results and unitary transformations analogous to those in \cite{ANRR}, that the wave operator $W_-$ in dimension $n = 4$ satisfies Equation \eqref{eq:wave-op-res}. In Section \ref{sec:decoupling} we decouple the resonant contribution $B_{res}$ by providing a suitable factorisation of the wave operator and show that Equation \eqref{eq:index} holds, giving a topological interpretation of Levinson's theorem in which resonances provide an integer contribution.

{\bf Acknowledgements} The authors thank Serge Richard for many illuminating conversations on the topic of resolvent expansions and resonances. The first author also acknowledges the support of an Australian Government Research Training Program (RTP) Scholarship. This project was supported by the ARC Discovery grant DP220101196.


\section{Preliminaries on scattering theory} \label{sec:prelim}
\label{sec:scats}

\subsection{Standing assumptions and notation}
\label{subsec:ass-notes}

Throughout this article we will consider the scattering theory on $\R^4$ associated to the operators
\[
H_0=-\sum_{j=1}^4\frac{\partial^2}{\partial x_j^2}=-\Delta
\quad\mbox{and}\quad
H=-\sum_{j=1}^4\frac{\partial^2}{\partial x_j^2}+V
\]
where the (multiplication operator by the) 
potential $V$ is real-valued and satisfies 
\begin{align}
\label{ass11}
|V(x)| &\leq C(1+|x|)^{-\rho}
\end{align}
for some $\rho > 12$. 
We denote the Schwartz space $\mathcal{S}(\R^4)$ and its dual $\mathcal{S}'(\R^4)$ and recall the weighted Sobolev spaces 
\[
H^{s,t}(\R^4)= \Big\{f \in \mathcal{S}'(\R^4): \norm{f}_{H^{s,t}} := \norm{(1+|x|^2)^{\frac{t}{2}} (\textup{Id}-\Delta)^\frac{s}{2} f} < \infty \Big\}
\] 
with index $s \in \R$ indicating derivatives and $t \in \R$ associated to decay at infinity \cite[Section 4.1]{amrein96}. With $\langle\cdot,\cdot\rangle$ the Euclidean inner product on $\R^4$, we denote the Fourier transform by
\[
\F:L^2(\R^4)\to L^2(\R^4),\qquad [\F f](\xi)=(2\pi)^{-2}\int_{\R^4}e^{-i\langle x,\xi\rangle}f(x)\,\d x.
\]
Note that the Fourier transform $\F$ is an isomorphism from $H^{s,t}$ to $H^{t,s}$ for any $s,t \in \R$.  We denote by $\mathcal{B}(\H_1,\H_2)$ and $\mathcal{K}(\H_1,\H_2)$ the bounded and compact operators from $\H_1$ to $\H_2$.
For $z \in \C \setminus \R$, we let
\[
R_0(z)=(H_0-z)^{-1},\qquad R(z)=(H-z)^{-1}
\]
and the boundary values of the resolvent are defined as
\begin{align}
R_0(\lambda \pm i0) &= \lim_{\eps \to 0}{R_0(\lambda \pm i \eps)}\quad\mbox{and}\quad R(\lambda \pm i0) = \lim_{\eps \to 0}{R(\lambda \pm i \eps)}.
\end{align}
The limiting absorption principle \cite[Theorem 4.2]{agmon75} tells us that these boundary values exist in $\mathcal{B}(H^{0,t},H^{2,-t})$ for any $t > \frac12$ and $\lambda \in (0,\infty)$. The operator $H_0$ has purely absolutely continuous spectrum, and in particular no kernel. The operator $H$ can have eigenvalues and for $V$ satisfying Assumption \eqref{ass11} with $\rho > 1$ we have that these eigenvalues are negative, or zero \cite[Theorem 6.1.1]{yafaev10} (see also \cite[Section 1]{kato59}). We let $P_0$ be the kernel projection of $H$, which may be zero.

The one-parameter unitary group of dilations on $L^2(\R^n)$ is
given on $f\in L^2(\R^n)$ by
\begin{align}\label{defn:dilation}
[U_n(t)f](x) &= \e^{\frac{nt}{2}} f(\e^t x),\qquad t\in\R.
\end{align}
We denote the self-adjoint generator of $U_n$ by $D_n$.
The generator of the group $(U_+(t))$ of dilations on the half-line $\R^+$ is denoted $D_+$
(which is $D_1$ restricted to the positive half-line). The generators of the dilation groups are given by
\begin{align}
D_+ &= \frac{y}{i} \od{}{y} + \frac{1}{2i}{ \textup{Id}},\qquad D_n=\sum_{j=1}^n\frac{x_j}{i}\frac{\partial}{\partial x_j}+\frac{n}{2i}{ \textup{Id}}.
\end{align}
Since each of $D_+,D_n$ generate one-parameter groups, we can recognise functions of these operators. For $D_+$ and $\vp:\R\to\C$ a bounded function whose Fourier transform has rapid decay, we have
\begin{align*}
[\vp(D_+) g](\rho) &= (2\pi)^{-\frac12} \int_\R{[\Fi^* \vp](t) \e^{\frac{t}{2}} g(\e^t \rho)\, \d t},
\end{align*}
with a similar formula for $D_n$.

Several Hilbert spaces recur, and we adopt the notation (following \cite[Section 2]{jensen81} which contains an excellent discussion on the relations between the spaces and operators we introduce here)
\[
\H = L^2(\R^4),\quad \mathcal{P} = L^2(\Sf^{3}),\quad \H_{spec} = L^2(\R^+, \mathcal{P}) \cong L^2(\R^+) \otimes \mathcal{P}.
\] 
Here $\H_{spec}$ provides the Hilbert space on which we can diagonalise the free Hamiltonian $H_0$.

Since $V$ is bounded, $H = H_0+V$ is self-adjoint with $\Dom(H) = \Dom(H_0)$. The wave operators 
\[
W_\pm=\mathop{\textup{s-lim}}_{t\to\pm\infty}e^{itH}e^{-itH_0}
\]
exist and are asymptotically complete if $\rho > 1$ \cite[Theorem 1.6.2]{yafaev10}. We will use the stationary scattering theory, which coincides with the time-dependent approach \cite[Section 5.3]{yafaev92} given our assumptions.
For suitable $f,g \in \H$ we can write \cite[Equation 0.6.9]{yafaev10}
\begin{align}\label{eq:statwaveop}
\langle W_\pm f, g \rangle &= \int_\R{\left(\lim_{\eps \to 0}{\frac{\eps}{\pi}\langle R_0(\lambda \pm i \eps) f, R(\lambda \pm i \eps) g \rangle } \right)\, \d \lambda}.
\end{align}

For our analysis of the wave operator, we describe the explicit unitaries which diagonalise our Hamiltonians. 

For $\lambda > 0$ the trace operator $\gamma(\lambda) : \mathcal{S}(\R^4) \to \mathcal{P}$ defined by
$[\gamma(\lambda) f](\omega):= f(\lambda^\frac12 \omega)$ defines a bounded operator and for each $s > \frac12$ and $t \in \R$ extends to a bounded operator on $H^{s,t}$ (see \cite[Theorem 2.4.3]{kuroda78}).

\begin{defn}
\label{def:diag}
For  $\lambda \in \R^+$, $s \in \R$ and $t > \frac12$ we define the operator 
\[
\Gamma_0(\lambda): H^{s,t} \to \mathcal{P}\quad \mbox{by}\quad
[\Gamma_0(\lambda) f](\omega) = 2^{-\frac12} \lambda^{\frac{1}{2}} [\F f](\lambda^\frac12 \omega)
\]
and the operator which diagonalises the free Hamiltonian $H_0$ as
\[
F_0: \H \to \H_{spec}\quad \mbox{by} \quad [F_0 f](\lambda,\omega) = [\Gamma_0(\lambda) f](\omega).
\] 
\end{defn}

\begin{lemma}[{\cite[p. 439]{jensen81}}]
\label{lem:eff-emm}
The operator $F_0$ is unitary. Moreover for $\lambda \in[0,\infty)$, $\omega\in \Sf^{3}$ and $f \in \H_{spec}$ we have
\[
[F_0H_0F_0^* f](\lambda,\omega)=\lambda f(\lambda,\omega)=:[Lf](\lambda,\omega).
\] 
\end{lemma}
Here we have defined the operator $L$ of multiplication by the spectral variable.

\subsection{Resolvent expansions, resonances and the scattering operator}\label{sec:resolvents}

Here we recall some known results regarding expansions related to the perturbed resolvent $R(z)$ in the limit $z \to 0$. Only the terms in the expansion relevant to later computations will be shown, however we note that higher terms can be computed explicitly \cite{erdogan14, jensen84, yajima22}. The low energy behaviour is sensitive to the presence of `zero-energy resonances'. These are essentially distributional solutions to $H\psi = 0$ which are not square-integrable but lie in some larger space. We will give the precise definition shortly. We define $v(x) = |V(x)|^\frac12$ and $u(x) = \textup{sign}(V(x))$ for $x \in \R$, so that $U,v$ are self-adjoint, $U$ is unitary and $V = vUv$. 

\begin{defn}
Suppose that $V$ satisfies the assumption \eqref{ass11} for some $\rho > 12$. Then we say that there exists an $s$-resonance (or just simply a resonance) for $H = H_0+V$ if there exists $\psi \in H^{0,-t}$ for some $t > 0$ such that $H\psi = 0$ (in the sense of distributions) and $\psi \notin \H$.
\end{defn}

Since $H$ has no positive eigenvalues, our assumption on the decay of the potential and the limiting absorption principle \cite[Theorem 4.2]{agmon75} guarantee that the norm limits
\begin{align*}
v R_0(\lambda \pm i 0) v &:= \lim_{\eps \to 0}{v R_0(\lambda \pm i \eps) v} \quad \textup{ and } \quad v R(\lambda \pm i0 )v := \lim_{\eps \to 0}{v R(\lambda \pm i \eps) v}
\end{align*}
exist in $\B(\H)$ and are continuous in $\lambda \in (0,\infty)$. For $\lambda, \eps > 0$ we have the equality
\begin{align*}
U - Uv R(\lambda \pm i \eps) v U &= \left(U+vR_0(\lambda \pm i \eps)v\right)^{-1},
\end{align*}
which implies the existence and continuity of $(0,\infty) \ni \lambda \mapsto \left(u+vR_0(\lambda \pm i 0)v\right)^{-1} \in \B(\H)$. Furthermore, we have $\lim_{\lambda \to \infty}{(U+v R_0(\lambda \pm i0)v)^{-1}} = U$, since $\lim_{\lambda \to \infty}{v R_0(\lambda \pm i0) v} = 0$ by \cite[Proposition 7.1.2]{yafaev10}. On the other hand, the existence of the limits $\lim_{\lambda \to 0}{(U+v R_0(\lambda \pm i0)v)^{-1}}$ depends on the existence of resonances and eigenvalues at zero. The problem has been described in detail in \cite{erdogan14} using the method of \cite{jensen01}. We outline the main results below.

Let $k \in \C \setminus \{ 0 \}$ with $\textup{Re}(k) \geq 0$ and define $\eta = \frac{1}{\ln{(k)}}$ (using the principal branch of the logarithm). Define the operator $M(k) = U + vR_0(-k^2) v$. Then we have the following result \cite[Proposition 5.3]{erdogan14}. 
\begin{thm}\label{thm:Mk-expansion}
If $V$ satisfies assumption \eqref{ass11} with $\rho > 12$ and $0 < |k| < k_0$ for sufficiently small $k_0$, the operator $M(k)^{-1}$ has the expansion
\begin{align*}
M(k)^{-1} &= k^{-2} D_2 + k^{-2} h(k) Q_1 \tilde{T}_1 Q_1 + h(k) A_1 + \eta^{-1} h(k) A_2 + \eta^{-1} Q_1 A_3 Q_1 + \tilde{R}(k),
\end{align*}
where $h(k) = (c_1+\eta^{-1} c_2)^{-1}$ for some $c_1, c_2 \in \C$, $\tilde{R}(k)$ is uniformly bounded for $0 < |k| < k_0$,  $D_2, \tilde{T}_1, A_1, A_2, A_3 \in \B(\H)$, $Q_1 \geq Q_2$ are orthogonal projections, $T_1 = Q_1 - Q_2$ is (at most) a rank one projection and
\begin{align*}
\tilde{T}_1 &= T_1 - T_1 C_{12} Q_2 - Q_2 C_{21} T_1 + Q_2 C_{22} Q_2
\end{align*}
for some $C_{12}, C_{21}, C_{22} \in \B(\H)$. 
\end{thm}
The projection $T_1$ is related to the existence of resonances. For a non-zero resonance $\psi$, the projection $T_1$ is given by $T_1 = \langle \cdot, Uv\psi \rangle Uv\psi$. The value of the constant $c_1$ is not important for our analysis, whilst we have
\begin{align}\label{eq:c2-value}
c_2 &= - \frac{\left( \langle v, Uv \psi \rangle\right)^{2}}{ (8 \pi^2)},
\end{align}
see for example \cite[Lemma 3.2.36]{alexander23} and \cite[Equation 2.2]{yajima22}. To complete the analysis, we also need the small energy behaviour of $\Gamma_0(\lambda)v$, which we can obtain from \cite[Lemma 2.12]{AR23}.

\begin{lemma}\label{lem:gamma-expansion}
Suppose that $V$ satisfies \eqref{ass11} for some $\rho > 12$. Then in $\B(\H, \mathcal{P})$ we have the expansion
\begin{align*}
\Gamma_0(\lambda) v&= \lambda^\frac12 \gamma_0 v + \lambda \gamma_1 v + O(\lambda^\frac{3}{2})
\end{align*}
as $\lambda \to 0$ in $\B(\H, \mathcal{P})$, where the operators $\gamma_0v, \gamma_1 v\in \B(\H,\mathcal{P})$ are given for $f \in \H$ by
\begin{align*}
[\gamma_j  vf](\omega) &= 2^{-\frac12}(2\pi)^{-2} \int_{\R^4}{(-i \langle x,\omega \rangle )^j v(x) f(x)\, \d x}.
\end{align*}
\end{lemma}

We can combine Theorem \ref{thm:Mk-expansion} and Lemma \ref{lem:gamma-expansion} to obtain an expansion of a related operator which will be useful in the next section. We use the convention that $\lambda > 0$, $k = -i \lambda^\frac12$ and $\eta = \left(\frac{\ln{(\lambda)}}{2} - i \frac{\pi}{2}\right)^{-1}$.

\begin{lemma}\label{lem:new-expansion}
Suppose $V$ satisfies the assumption \eqref{ass11} for some $\rho > 12$. Then as $\lambda \to 0$ we have the expansion
\begin{align*}
\left(U+v R_0(\lambda+i0) v \right)^{-1} v \Gamma_0(\lambda)^* &= \lambda^{-\frac12} h(\lambda^\frac12) (T_1 - Q_2 C_{21} T_1) v \gamma_0^* + D_2 v \gamma_1^* + O(\eta)
\end{align*}
in $\B(\mathcal{P}, \H)$.
\end{lemma}
\begin{proof}
The identity $\gamma_0 v Q_2 = 0 = Q_2 v \gamma_0^*$ \cite[Lemma 7.2]{erdogan14} yields $\tilde{T}_1 v \gamma_0^* = (T_1 - Q_2 C_{21} T_1) v\gamma_0^*$. Observing that as $\lambda \to 0$, $\eta \to 0$ and $h(\lambda^\frac12) \to 0$, we multiply the expansions of Theorem \ref{thm:Mk-expansion} and Lemma \ref{lem:gamma-expansion} to obtain the statement.
\end{proof}

We summarise below some useful properties of the scattering matrix \cite[Proposition 1.8.1 and Proposition 8.1.9]{yafaev10}, \cite[Theorem 2.15 and Corollary 3.10]{AR23}.
\begin{thm}\label{thm: stationary scattering operator}
Suppose that $V$ satisfies the assumption \eqref{ass11} for some $\rho > 1$. The scattering matrix $S(\lambda)$ is given for all $\lambda \in \R^+$ by the equation
\begin{align}\label{eq: scat matrix defn}
S(\lambda) &= \textup{Id} - 2\pi i \Gamma_0(\lambda)v\left(U+vR_0(\lambda + i0) v\right)^{-1} v \Gamma_0(\lambda)^*.
\end{align}
For each $\lambda\in\R^+$, the operator $S(\lambda)$ is unitary in $\mathcal{P}=L^2(\Sf^{3})$ and depends continuously (in the sense of norm) on $\lambda \in \R^+$. Furthermore, if $\rho > 12$ then $S(\lambda)-\textup{Id} \in \mathcal{L}^1(\H)$, is differentiable in the norm of $\mathcal{L}^1(\H)$ and we have $S(0) = \textup{Id}$ and $\lim_{\lambda \to \infty}{S(\lambda)} = \textup{Id}$, where the limit is taken in $\B(\mathcal{P})$.
\end{thm}


\section{The form of the wave operator}
\label{sec:wave-op}

In this section, we analyse the wave operator $W_-$ in the space $\H_{spec}$ (the spectral representation for $H_0$). Much of the analysis is identical to the two dimensional case in \cite{ANRR}, to which we refer for many details. The main result in this section is the following.

\begin{thm}\label{thm:main-wave}
Suppose that $V$ satisfies the estimate \eqref{ass11} for some $\rho > 12$. Then the wave operator $W_-$ is given by
\begin{align}
W_- &= \textup{Id}+ \vp(D_4)(S-\textup{Id}) + B_{res} + K,
\end{align}
where $\vp:\R \to \C$ is given by
\begin{align*}
\vp(x) &= \frac12 \left( 1+\tanh{(\pi x)} -i \cosh{(\pi x)}^{-1} \right),
\end{align*}
$B_{res}$ is a bounded operator depending on the existence of resonances and $K$ is a compact operator.
\end{thm}

In Lemma \ref{lem:B-res} we will determine an explicit expression for the operator $B_{res}$ in a different representation. We begin our analysis of the wave operator with the stationary representation for the wave operator $W_-$. For suitable $f,g \in \H_{spec}$ we have
\begin{align}\label{eq:stat-wave-op}
&\langle F_0(W_--\textup{Id}) F_0^* f,g \rangle \\
&= - \int_\R{\lim_{\eps \to 0} \left( \int_0^\infty{ \left\langle \Gamma_0(\mu) v \left(U+vR_0(\lambda+i\eps) v \right)^{-1} v F_0^* \delta_\eps(L-\lambda) f,  \frac{g(\mu)}{\mu-\lambda+i\eps} \right\rangle_{\mathcal{P}} \, \d \mu} \right) \, \d \lambda} \nonumber,
\end{align}
where we have defined
\begin{align*}
\delta_\eps(L-\lambda) &:= \frac{\eps}{\pi} (L-\lambda+i\eps)^{-1}(L-\lambda- i \eps)^{-1}.
\end{align*}
Next we write $\Gamma_0(\mu)v=\Gamma_0(\mu)vQ_1+\Gamma_0(\mu)vQ_1^\perp$, which then gives two terms in Equation \eqref{eq:stat-wave-op}. The term with $Q_1^\perp$ has no singularity at zero energy by Lemma \ref{lem:new-expansion} and has been considered in \cite{AR23}. A similar expansion has been considered in great detail in \cite{ANRR} in dimension $n = 2$.

We now carefully analyse Equation \eqref{eq:stat-wave-op}. Fix a function $\chi_1 \in C(\R^+; [0,1])$ with $\chi_1(\lambda) = 0$ for $\lambda < \frac{1}{4}$ and $\chi_1(\lambda) = 1$ for $\lambda > \frac{3}{4}$ and let $\chi_2 = 1- \chi_1$. Let also the function $\chi_0 \in C(\R^+;[0,1])$ be such that $\chi_0(\lambda) = 1$ for $\lambda < \frac{3}{4}$ and $\chi_0(\lambda) = 0$ for $\lambda > \frac{7}{8}$. For $f \in \H_{spec}$ and $\lambda, \eps > 0$ we define
\begin{align*}
\tilde{\psi}_\eps(\lambda) &:= Q_1 \left(U+vR_0(\lambda+i\eps) v \right)^{-1} v F_0^* \delta_\eps(L-\lambda) \chi_2(L) f.
\end{align*}

We make the following decomposition of Equation \eqref{eq:stat-wave-op}. Define
\begin{align*}
&R_0 \\
&= \int_0^\infty{\lim_{\eps \to 0} \left( \int_0^\infty{ \left\langle \frac{\Gamma_0(\mu) v Q_1}{\mu-\lambda-i\eps} Q_1 \left(U+v R_0(\lambda + i \eps) v \right)^{-1} v F_0^* \delta_\eps(L-\lambda) \chi_1(L) f, g(\mu) \right\rangle \, \d \mu} \right) \, \d \lambda}, \\
&R_1 = \int_0^\infty{\lim_{\eps \to 0} \left( \int_0^\infty{ \left\langle \Gamma_0(\mu) v Q_1\chi_2(\mu) \mu^{-\frac12}(\mu^\frac12-\lambda^\frac12) (\mu-\lambda-i\eps)^{-1} v F_0^* \tilde{\psi}_\eps(\lambda), g(\mu) \right\rangle \, \d \mu} \right) \, \d \lambda}, \\
&R_2 = \int_0^\infty{\lim_{\eps \to 0} \left( \int_0^\infty{ \left\langle \Gamma_0(\mu) v Q_1\chi_2(\mu) \mu^{-\frac12}\lambda^\frac12 (\mu-\lambda-i\eps)^{-1} v F_0^* \tilde{\psi}_\eps(\lambda), g(\mu) \right\rangle \, \d \mu} \right) \, \d \lambda}, \\
&R_3 = \int_0^\infty{\lim_{\eps \to 0} \left( \int_0^\infty{ \left\langle \Gamma_0(\mu) v Q_1\chi_2(\mu) \mu^{-1}(\mu-\lambda) (\mu-\lambda-i\eps)^{-1} v F_0^* \tilde{\psi}_\eps(\lambda), g(\mu) \right\rangle \, \d \mu} \right) \, \d \lambda},\\
&R_4 = \int_0^\infty{\lim_{\eps \to 0} \left( \int_0^\infty{ \left\langle \Gamma_0(\mu) v Q_1\chi_2(\mu)\mu^{-1} \lambda (\mu-\lambda-i\eps)^{-1} v F_0^* \tilde{\psi}_\eps(\lambda), g(\mu) \right\rangle \, \d \mu} \right) \, \d \lambda},  \textup{ and }\\
&R_5 = \int_0^\infty{\lim_{\eps \to 0} \left( \int_0^\infty{ \left\langle \Gamma_0(\mu) v Q_1^\perp\chi_2(\mu) \mu^{-1}(\mu-\lambda) (\mu-\lambda-i\eps)^{-1} v F_0^* \tilde{\psi}_\eps(\lambda), g(\mu) \right\rangle \, \d \mu} \right) \, \d \lambda}.
\end{align*}
By construction we have the equality
\begin{align*}
-\langle F_0(W_--\textup{Id})F_0^* f, g \rangle &= R_0+R_1+R_2+R_3+R_4+R_5.
\end{align*}

Combining \cite[Corollary 4.2]{ANRR} and \cite[Theorem 3.1]{AR23} we have the following.

\begin{lemma}\label{lem:many-Rs}
Suppose that $V$ satisfies \eqref{ass11} for some $\rho > 12$. Then we have
\begin{align*}
R_0 + R_2 + R_3 + R_4 + R_5 &= -\langle \left(\textup{Id} + \vp(D_4)(S-\textup{Id}) + K \right) f, g \rangle
\end{align*}
for a compact operator $K$.
\end{lemma}
It remains to analyse the term $R_1$, which contains the worst singularity of $M(k)^{-1}$. For this analysis we require some additional preparatory results. We consider the unitary transformation $\mathcal{U}: L^2(\R^+) \to L^2(\R)$ defined for $f \in L^2(\R^+)$ and $x \in \R$ by
\begin{align}\label{eq:U-defn}
[\mathcal{U} f](x) &= 2^{-\frac12} \e^{-x} f(\e^{-2x}).
\end{align}
We also introduce the integral operator $\Xi: L^2(\R^+) \to L^2(\R^+)$ with kernel given by
\begin{align}\label{eq:Xi-defn}
\Xi(\mu,\lambda) &= \chi_0(\mu) \frac{1}{\mu^\frac12 + \lambda^\frac12} \frac{1}{\lambda^\frac12 \ln{(\lambda)}} \chi_0(\lambda).
\end{align}
We can identify $\Xi$ as a function of the operators $X$ of position and $D_1$ of dilation in $\R$. The following is \cite[Lemma 3.6]{ANRR}.
\begin{lemma}\label{lem:Xi-rep}
Define $\Xi: L^2(\R^+) \to L^2(\R^+)$ by Equation \eqref{eq:Xi-defn}. Then we have in $L^2(\R)$ the equality
\begin{align*}
\mathcal{U} \Xi \mathcal{U}^* &= - \chi_0(\e^{-2X}) \frac{2}{1+2iD_1} \chi_0(\e^{-2X}) + K_1,
\end{align*}
where $D_1$ is the generator of dilations in $L^2(\R)$, $X$ is the operator of multiplication by the variable in $L^2(\R)$ and $K_1$ is a compact operator.
\end{lemma}
To analyse the term $R_1$ we need to introduce two more operators, analogues of which have been studied in \cite[Lemmas 3.1 and 3.5]{ANRR}.
\begin{lemma}\label{lem:B-N-properties}
Define the maps $B:\R^+ \to \mathcal{K}(\mathcal{P},\H)$ and $N: \R^+ \to \B(\H, \mathcal{P})$ for $\lambda, \mu \in \R^+$ by
\begin{align*}
B(\lambda) &= \chi_2(\lambda) \lambda^\frac12 \ln{(\lambda)} Q_1 \left(U+vR_0(\lambda+i0) v\right)^{-1} v \Gamma_0(\lambda)^*, \quad \textup{ and } \\
N(\mu) &= \Gamma_0(\mu) \chi_2(\mu) \mu^{-\frac12} v Q_1.
\end{align*}
The multiplication operator defined by $B$ extends continuously to  $\B(\H_{spec}, L^2(\R^+,\H))$ and the multiplication operator defined by $N$ extends continuously to $\B(L^2(\R^+, \H), \H_{spec})$.
\end{lemma}
By an identical proof to \cite[Lemma 3.7]{ANRR} we obtain the following.
\begin{lemma}\label{lem:B-res}
The term $R_1$ can be written as $\langle (B_{res}+K) f, g \rangle$ with $B_{res} = N\Xi B$ and $K$ compact.
\end{lemma}
Collecting together the results of Lemmas \ref{lem:many-Rs} and \ref{lem:B-res} we can now prove Theorem \ref{thm:main-wave}.
\begin{proof}[Proof of Theorem \ref{thm:main-wave}]
By construction we have the equality
\begin{align*}
-\langle F_0(W_--\textup{Id})F_0^* f, g \rangle &= R_0+R_1+R_2+R_3+R_4+R_5.
\end{align*}
By Lemma \ref{lem:many-Rs} we have
\begin{align*}
R_0 + R_2 + R_3 + R_4 + R_5 &= -\langle \left(\textup{Id} + \vp(D_4)(S-\textup{Id}) + K \right) f, g \rangle.
\end{align*}
By Lemma \ref{lem:B-res} we have $R_1 = \langle (B_{res}+K) f, g \rangle$ with $B_{res} = N\Xi B$ and $K$ compact. Adding these results together gives 
\begin{align}\label{eq:Wave-op-expression}
F_0(W_--\textup{Id})F_0^* &= \varphi(D_4)(S-\textup{Id}) -N \Xi B + K
\end{align}
for a compact operator $K$, from which the statement follows.
\end{proof}
 As suggested by the analysis of Lemma \ref{lem:Xi-rep}, we investigate Equation \eqref{eq:Wave-op-expression} in the Hilbert space $L^2(\R, \mathcal{P})$ by using the unitary map $\mathcal{U}$ of Equation \eqref{eq:U-defn}. For any multiplication operator $M$ defined by $\R^+ \ni \lambda \mapsto M(\lambda)$ and $f \in L^2(\R,\mathcal{P})$ we have
\begin{align*}
[\mathcal{U}M\mathcal{U}^* f](x) &= M\left(\e^{-2x}\right)  f(x) = \left[M(\e^{-2X}) f\right](x).
\end{align*}
If we consider the dilation group $(U_+(t))_{t\in\R}$ we obtain
\begin{align*}
[\mathcal{U} U_+(t) \mathcal{U}^* f](x) &= f\left(x-\frac12 t\right) = \left[\e^{-it \frac12 D} f\right](x),
\end{align*}
where $D = -i \frac{\d}{\d x}$. Combining this information we find
\begin{align*}
&\mathcal{U} F_0(W_--\textup{Id})F_0^* \mathcal{U}^* \\
&= \left( \frac12(\textup{Id}-\tanh{\left(\frac{\pi}{2} D \right)} - i \cosh{\left(\frac{\pi}{2} D \right)}^{-1} \right)(\tilde{S}(X)-\textup{Id})-\tilde{N}(X) \left( \frac{2}{1+2iD_1} \right) \tilde{B}(X) + K,
\end{align*}
where $\tilde{S}(X) = S\left(\e^{-2X}\right)$, $\tilde{N}(X) = N\left(\e^{-2X}\right)$, $\tilde{B}(X) = B\left(\e^{-2X}\right)$ and $K$ is a compact operator. We note that the three generators $X,D, D_1$, the position operator, the generator of translations and the generator of dilations are all present in this expression. A $C^*$-algebra generated by continuous functions of these three operators has been considered in \cite[Chapter V]{cordes79}. The algebra of Cordes is constructed on the Hilbert space $L^2(\R^+)$, so we need an additional unitary to use this framework, namely the decomposition of $L^2(\R)$ into even and odd functions.

Define $\mathcal{V} : L^2(\R) \to L^2(\R^+,\C^2)$ by
\begin{align*}
[\mathcal{V} f](y) &= 2^\frac12 \begin{pmatrix} \frac{f(y)+f(-y)}{2} \\ \frac{f(y)-f(-y)}{2} \end{pmatrix} := 2^\frac12 \begin{pmatrix} f_e \\ f_o \end{pmatrix}
\end{align*}
for $f \in L^2(\R)$. The adjoint of $\mathcal{V}$ is defined for $g_1,g_2 \in L^2(\R^+)$ by
\begin{align*}
\left[ \mathcal{V}^* \begin{pmatrix} g_1 \\ g_2 \end{pmatrix} \right](x) &= 2^{-\frac12} \left[ g_1(|x|) + \textup{sign}(x) g_2(|x|) \right].
\end{align*}
If $m:\R \to \C$ we find
\begin{align*}
\mathcal{V} m(X) \mathcal{V}^* &= \begin{pmatrix} m_e(L) & m_o(L) \\ m_o(L) & m_e(L) \end{pmatrix},
\end{align*}
whilst
\begin{align*}
\mathcal{V} m(D_1) \mathcal{V}^* &= \begin{pmatrix} m(D_+) & 0 \\ 0 & m(D_+) \end{pmatrix}.
\end{align*}
In order to consider $\mathcal{V} m(D)\mathcal{V}^*$, let us denote by $\F_1$ the usual unitary Fourier transform
in $L^2(\R)$, and let $\F_{ N}$, $\F_{ D}$ be the unitary cosine and sine transforms on $L^2(\R_+)$, respectively.  The index $ N$ and $ D$ are related to the Neumann Laplacian and the Dirichlet Laplacian in $L^2(\R_+)$, which are diagonalised by $\F_{ N}$ and $\F_{ D}$, respectively.  
Note also that the operators $\F_N$ and $\F_D$ correspond to their own inverse.
It is then easily checked that
\begin{align*}
\mathcal{V} \F_1\mathcal{V}^* &=
\begin{pmatrix}
\F_{ N} & 0\\
0 & i\F_{ D}
\end{pmatrix}.
\end{align*}
In addition, by a straightforward computation one gets
\begin{align*}
\mathcal{V} m(D)\mathcal{V}^* = \mathcal{V} \F_1^* m(X) \F_1 \mathcal{V}^*
&= 
\begin{pmatrix}
\F_{ N} m_{ e}(L)\F_{ N} &  -i \F_{ N} m_{ o}(L)\F_{ D}\\
i \F_{ D}m_{ o}(L)\F_{ N} & \F_{ D} m_{ e}(L) \F_{ D}
\end{pmatrix}.
\end{align*}
For the final step, we recall that the Neumann Laplacian satisfies $-\Delta_{ N}:=\F_{ N} L^2 \F_{ N}$, and that 
\begin{align*}
i\F_{ N} \F_{ D} &= -\tanh(\pi D_+)+ i \cosh(\pi D_+)^{-1}=:\phi(D_+).
\end{align*}
We refer for example to \cite[Proposition 4.13]{DR} for a proof of the above equality.
Then, we obtain
\begin{align*}\label{eq:D-unitary}
\mathcal{V} m(D)\mathcal{V}^* &= \mathcal{V} \F_1^* m(X) \F_1 \mathcal{V}^*
 = 
\begin{pmatrix}
m_{ e}\big(\sqrt{-\Delta_{ N}}\big) &  - m_{ o}\big(\sqrt{-\Delta_{ N}}\big)
\phi(D_+)\\
-\overline{\phi}(D_+) m_{ o}\big(\sqrt{-\Delta_{ N}}\big) 
&\  \overline{\phi}(D_+) m_{ e}\big(\sqrt{-\Delta_{ N}}\big) \phi(D_+)
\end{pmatrix}.
\end{align*}
Combining these we can obtain the image of the wave operator in $L^2(\R^+,\C^2)$.

\begin{lemma}\label{lem:wave-op-transformed}
The expression $\mathcal{V}\mathcal{U}F_0 W_-F_0^*\mathcal{U}^* \mathcal{V}^*$ is given by
\begin{align*}
&
\begin{pmatrix} 
1 & 0 \\ 0 & 1
\end{pmatrix}
+\frac{1}{2}
\begin{pmatrix} 
1 &  \overline{\phi}\left( \frac12 \sqrt{-\Delta_{N}} \right) \phi(D_+) \\ 
\phi(D_+)\overline{ \phi}\left( \frac12 \sqrt{-\Delta_{N}} \right)  & 1
\end{pmatrix}
\begin{pmatrix} 
\tilde{S}_{ e}(L)-1 & \tilde{S}_{ o}(L)  \\ 
\tilde{S}_{ o}(L) & \tilde{S}_{ e}(L) - 1
\end{pmatrix}
\\
& \qquad +\begin{pmatrix} 
\tilde{N}_{ e}(L) & \tilde{N}_{ o}(L)  \\ 
\tilde{N}_{ o}(L) & \tilde{N}_{ e}(L) 
\end{pmatrix}
 \begin{pmatrix} 
\frac{2}{1+2iD_+} & 0  \\ 
0 & \frac{2}{1+2iD_+}
\end{pmatrix}
\begin{pmatrix} 
\tilde{B}_{ e}(L) & \tilde{B}_{ o}(L)  \\ 
\tilde{B}_{ o}(L) & \tilde{B}_{ e}(L) 
\end{pmatrix}  + K \nonumber
\end{align*}
with $K$ compact.
\end{lemma}
We now determine the precise contribution of resonances to the above equality.
\begin{lemma}\label{lem:N-B-zero}
We have the equality
\begin{align*}
N(0) B(0) &= - P_s,
\end{align*}
where $P_s = 0$ if there does not exist a resonance and $P_s$ is the orthogonal projection onto the spherical harmonics of order zero in $\mathcal{P}$ if there exists a resonance.
\end{lemma}
\begin{proof}
The case when there are no resonances is clear since $T_1=0$, so suppose there does exist a resonance. We use Lemma \ref{lem:new-expansion} to see that
\begin{align*}
B(0) &= \lim_{\lambda \to 0}  \chi_2(\lambda) \lambda^\frac12 \ln{(\lambda)} Q_1 \left(U+vR_0(\lambda+i0) v\right)^{-1} v \Gamma_0(\lambda)^* \\
&= \lim_{\lambda \to 0}  \chi_2(\lambda) \lambda^\frac12 \ln{(\lambda)} \left( \lambda^{-\frac12} h(\lambda^\frac12) (T_1 - Q_2 C_{21} T_1) v \gamma_0^* + D_2 v \gamma_1^* + O(\eta) \right) \\
&=  \left( \lim_{\lambda \to 0} \ln{(\lambda)} h(\lambda^\frac12) \right) (T_1-Q_2 C_{21} T_1) v \gamma_0^* \\
&= \left( \lim_{\lambda \to 0} \frac{\ln{(\lambda)}}{ c_1+\left( \frac{\ln ( \lambda)}{2} - i \frac{\pi}{2} \right) c_2} \right) (T_1-Q_2 C_{21} T_1) v \gamma_0^* \\
&= \frac{2}{c_2}(T_1-Q_2 C_{21} T_1 ) v \gamma_0^*.
\end{align*}
The small energy behaviour of $\Gamma_0(\mu)v$ and the definition of $N$ gives
$
N(0) = \gamma_0 v Q_1.
$
Observing the relation $\gamma_0 v Q_2 = 0$ we multiply to obtain
\begin{align*}
N(0) B(0) &= \frac{2}{c_2} \gamma_0 v T_1 v \gamma_0^*.
\end{align*}
Recall from the discussion around Theorem \ref{thm:Mk-expansion} that for a non-zero resonance $\psi$ we have $T_1 = \langle \cdot, Uv\psi \rangle Uv\psi$. We find for $f \in \mathcal{P}$ and $\omega \in \mathbb{S}^3$ that
\begin{align*}
[\gamma_0 v T_1 v \gamma_0^* f](\omega) &= 2^{-\frac12} (2\pi)^{-2} \int_{\R^4}{v(x) [T_1 v \gamma_0^* f](x)\, \d x} \\
&= 2^{-\frac12} (2\pi)^{-2} \int_{\R^4}{v(x) \langle v \gamma_0^*f ,  Uv \psi \rangle U(x) v(x) \psi(x) \, \d x} \\
&= 2^{-1} (2\pi)^{-4} \left(\langle v, Uv\psi \rangle \right)^2 \int_{\mathbb{S}^3}{f(\theta) \, \d \theta}.
\end{align*}
Now note that $\textup{Vol}(\mathbb{S}^3) = 2 \pi^2$ to see that 
\begin{align*}
[\gamma_0 v T_1 v \gamma_0^* f](\omega) &= 2^{-1} (2\pi)^{-4} \left(\langle v, Uv\psi \rangle \right)^2 (2 \pi^2) \left( \frac{1}{\textup{Vol}(\mathbb{S}^3)} \int_{\mathbb{S}^3}{f(\theta) \, \d \theta} \right) \\
&= \frac{\left(\langle v, Uv\psi \rangle \right)^2 }{16 \pi^2} [P_s f](\omega),
\end{align*}
where $P_s$ denotes the projection onto the spherical harmonics of order zero in $\mathcal{P}$. Recalling the value of $c_2$ from Equation \eqref{eq:c2-value} we find the statement of the lemma.
\end{proof}


\section{Index theory and Levinson's theorem}\label{sec:decoupling}

In \cite[Section V.7]{cordes79}, the following $C^*$-subalgebra of $\B(L^2(\R^+))$ was introduced:
\begin{align*}
E &:= C^*\big(a_i(D_+)b_i(L)c_i(-\Delta_{N}): a_i \in C([-\infty,\infty]), b_i, c_i \in C([0,\infty])\big).
\end{align*}
It is shown in \cite[Theorem V.7.3]{cordes79} that the quotient algebra $E \slash \mathcal{K}(L^2(\R^+))$ is isomorphic to $C(\hexagon)$, the set of continuous functions defined on the edges of a hexagon.

As a consequence of the results of Cordes, we have the short exact sequence
$$
0 \longrightarrow \K\big(L^2(\R^+;\mathcal{P})^{\oplus2}\big) \longrightarrow
\big(M_2(E)\otimes \K(\mathcal{P})\big)^\sim \stackrel{q}{\longrightarrow} 
 \big(M_2\big(C(\hexagon)\big)\otimes \K(\mathcal{P})\big)^\sim \longrightarrow 0
$$
of $C^*$-algebras, and the corresponding $6$ term exact sequence in $K$-theory. In particular, we have $K_0\left( \K(L^2(\R^+, \mathcal{P})^{\oplus2}) \right) \cong \Z$ and $K_1\left( (M_2(C(\hexagon)) \otimes \K(\mathcal{P}))^\sim \right) \cong \Z$.

This picture will allow us to interpret Levinson's theorem as a topological result. Our interest in the $C^*$-algebra $E$ is that the wave operator is a product of functions of $D_+, L, -\Delta_N$, as shown in Lemma \ref{lem:wave-op-transformed}, and these functions are continuous and have limits at their endpoints. As such, the wave operator is an element of $(M_2(E) \otimes \mathcal{K}(\mathcal{P}))^\sim$.

%

We can then consider the image of the wave operator under the quotient map 
\begin{align*}
q&: (M_2(E) \otimes \mathcal{K}(\mathcal{P}))^\sim \to \left( M_2(C(\hexagon)) \otimes \mathcal{K}(\mathcal{P}) \right)^\sim,
\end{align*}
with kernel $M_2(\mathcal{K}(L^2(\R^+))) \otimes \mathcal{K}(\mathcal{P})$.

%

\begin{prop}\label{prop:quotient}
The operator $\mathcal{V} \mathcal{U} F_0 W_- F_0^* \mathcal{U}^* \mathcal{V}^*$ is an element of $\left(M_2(E) \otimes \mathcal{K}(\mathcal{P}) \right)^\sim$. Hence, up to unitary equivalence, the image of the wave operator $W_-$ through the quotient map $q$ is a continuous function
\begin{align*}
\Gamma &:= (\Gamma_1, \Gamma_2, \Gamma_3, \Gamma_4, \Gamma_5, \Gamma_6) : \hexagon \to M_2\ox \K(\mathcal{P})^\sim
\end{align*}
even at the vertices of the hexagon.
The restrictions to the edges (oriented as indicated) are
\begin{align*}
\Gamma_1(s) &= \begin{pmatrix} 1 & 0 \\ 0 & 1 \end{pmatrix} + \frac12 (S(1)-1) \begin{pmatrix} 1 & \phi(s) \\ \overline{\phi(s)} & 1 \end{pmatrix}, \quad & s \in [-\infty,\infty], \\
\Gamma_2(\ell) &= \begin{pmatrix} 1 & 0 \\ 0 & 1 \end{pmatrix}  + \frac12 (S(\e^{2\ell})-1) \begin{pmatrix} 1 & -1 \\ -1 & 1 \end{pmatrix}, \quad & \ell \in [0,\infty], \\
\Gamma_3(\xi) &= \begin{pmatrix} 1 & 0 \\ 0 & 1 \end{pmatrix}, \quad & \xi \in [\infty, 0], \\
\Gamma_4(s) &= \begin{pmatrix} 1 & 0 \\ 0 & 1 \end{pmatrix}  -\frac12 \frac{2}{1+2is}P_s \begin{pmatrix} 1 & 1 \\ 1 & 1 \end{pmatrix}, \quad & s \in [\infty, -\infty], \\
\Gamma_5(\xi) &= \begin{pmatrix} 1 & 0 \\ 0 & 1 \end{pmatrix} , \quad & \xi \in [0,\infty], \\
\Gamma_6(\ell) &= \begin{pmatrix} 1 & 0 \\ 0 & 1 \end{pmatrix}  + \frac12 (S(\e^{-2\ell})-1) \begin{pmatrix} 1 & 1 \\ 1 & 1 \end{pmatrix}, \quad & \ell \in [\infty,0].
\end{align*}
\end{prop}
\begin{proof}
The continuity and existence of the limits of the endpoints of the components of $\mathcal{V}\mathcal{U} F_0 W_- F_0^* \mathcal{U}^* \mathcal{V}^*$ have already been established in Lemmas \ref{lem:Xi-rep} and \ref{lem:B-N-properties}. Note also that by Theorem \ref{thm: stationary scattering operator} we have $\lambda \mapsto S(\lambda)-\textup{Id}$ belongs to $C_0(\R^+, \mathcal{K}(\mathcal{P}))$. Thus we find $\mathcal{V}\mathcal{U} F_0 W_- F_0^* \mathcal{U}^* \mathcal{V}^* \in \left(M_2(E) \otimes \mathcal{K}(\mathcal{P}) \right)^\sim$. 

We now consider the image of the operator $\mathcal{V}\mathcal{U} F_0 W_- F_0^* \mathcal{U}^* \mathcal{V}^* $ under the quotient map. For $\Gamma_1$, we note that $\tilde{S}_e(0) = S(1)$ and $\tilde{S}_o(0) = 0$. Since $\chi_2$ vanishes at infinity we observe also that $\tilde{N}(0) = \tilde{B}(0) = 0$. For $\Gamma_2$, we note that $\lim_{s \to \infty}{\phi(s)} = -1$ to obtain the result. For $\Gamma_6$, we use the observation that $\lim_{s \to -\infty}{\phi(s)} = 1$. For $\Gamma_3$ and $\Gamma_5$, we recall that $S(0) = \lim_{\lambda \to \infty}(S(\lambda)) = \textup{Id}$ to see that $\lim_{\lambda \to \pm \infty}{\tilde{S}_e(\ell)} = 1$ while $\lim_{\lambda \to \pm \infty} \tilde{S}_o(\ell) = 0$. Finally, for $\Gamma_4$ we note that $\lim_{\ell \to \infty}{\tilde{N}(\ell)} = N(0)$ and $\lim_{\ell \to \infty}{\tilde{B}(\ell)} = B(0)$. Thus we find
\begin{align*}
\lim_{\ell \to \infty} \tilde{N}_e(\ell) &= \lim_{\ell \to \infty} \tilde{N}_o(\ell) = \frac12 N(0)
\end{align*}
and
\begin{align*}
\lim_{\ell \to \infty} \tilde{B}_e(\ell) &= \lim_{\ell \to \infty} \tilde{B}_o(\ell) = \frac12 B(0).
\end{align*}
An application of Lemma \ref{lem:N-B-zero} then completes the proof.
\end{proof}

As a result of \cite[Lemma 5.1]{ANRR} we have the following.
\begin{lemma}
The pointwise determinant of each component $\Gamma_j$ exists and they are given by
\begin{align*}
\det(\Gamma_1(s)) &= \det(S(1)), \quad & s \in [-\infty,\infty], \\
\det(\Gamma_2(\ell)) &= \det(S(\e^{2\ell})), \quad & \ell \in [0,\infty], \\
\det(\Gamma_3(\xi)) &= 1, \quad & \xi \in [\infty, 0], \\
\det(\Gamma_4(s)) &= \left( \frac{2is-1}{2is+1} \right)^{\dim(P_s)}, \quad & s \in [\infty, -\infty], \\
\det(\Gamma_5(\xi)) &= 1 , \quad & \xi \in [0,\infty], \\
\det(\Gamma_6(\ell)) &= \det(S(\e^{-2\ell})), \quad & \ell \in [\infty,0].
\end{align*}
\end{lemma}

We now decouple the resonant contribution to the wave operator. We do this in order to isolate the numerical contribution of resonances to Levinson's theorem.

Define the operators $W_S, W_R  \in \B(\H)$ by the equalities
\begin{align}
F_0(W_S-\textup{Id}) F_0^* &= \frac12 \varphi\left(-\frac12 D_+ \right)(S(L)-\textup{Id}).
\end{align}
and
\begin{align*}
F_0(W_R-\textup{Id}) &= -N \Xi B.
\end{align*}
Then we have the following.
\begin{lemma}
The operators $W_S$ and $W_R$ are Fredholm and we have the equality
\begin{align}
W_- &= W_S W_R + K
\end{align}
for a compact operator $K$. In particular, we have
\begin{align}
\textup{Index}(W_-) &= \textup{Index}(W_S) + \textup{Index}(W_R).
\end{align}
\end{lemma}
\begin{proof}
To see that $W_S$ defines a Fredholm operator, it is sufficient to note that $W_{S^*}$ defines an inverse for $W_S$ up to compacts (see \cite[Lemma 5.3]{ANRR}). As in the proof of \cite[Proposition 5.5]{ANRR} we have that $W_R^*$ defines an inverse for $W_R$ modulo compacts. We next observe that by Equation \eqref{eq:Wave-op-expression} we have the equality
\begin{align}\label{eq:factorisation}
W_- &= W_S + (W_R-\textup{Id})+K = W_S(\textup{Id} + W_S^* \left(W_R-\textup{Id} \right)) + \tilde{K}
\end{align}
with $K, \tilde{K}$ compact operators. Denote by $\Gamma_{S,j}$ and $\Gamma_{R,j}$ the components of $\mathcal{V}\mathcal{U} F_0 W_S F_0^* \mathcal{U}^* \mathcal{V}^*$ and $\mathcal{V}\mathcal{U} F_0 W_R F_0^* \mathcal{U}^* \mathcal{V}^*$ under the quotient map $q$. A proof similar to Proposition \ref{prop:quotient} shows that
\begin{align*}
\Gamma_{S,1}(s) &= \begin{pmatrix} 1 & 0 \\ 0 & 1 \end{pmatrix} + \frac12 (S(1)-1) \begin{pmatrix} 1 & \phi(s) \\ \overline{\phi(s)} & 1 \end{pmatrix}, \quad & s \in [-\infty,\infty], \\
\Gamma_{S,2}(\ell) &= \begin{pmatrix} 1 & 0 \\ 0 & 1 \end{pmatrix}  + \frac12 (S(\e^{2\ell})-1) \begin{pmatrix} 1 & -1 \\ -1 & 1 \end{pmatrix}, \quad & \ell \in [0,\infty], \\
\Gamma_{S,3}(\xi) &= \begin{pmatrix} 1 & 0 \\ 0 & 1 \end{pmatrix}, \quad & \xi \in [\infty, 0], \\
\Gamma_{S,4}(s) &= \begin{pmatrix} 1 & 0 \\ 0 & 1 \end{pmatrix} , \quad & s \in [\infty, -\infty], \\
\Gamma_{S,5}(\xi) &= \begin{pmatrix} 1 & 0 \\ 0 & 1 \end{pmatrix} , \quad & \xi \in [0,\infty], \\
\Gamma_{S,6}(\ell) &= \begin{pmatrix} 1 & 0 \\ 0 & 1 \end{pmatrix}  + \frac12 (S(\e^{-2\ell})-1) \begin{pmatrix} 1 & 1 \\ 1 & 1 \end{pmatrix}, \quad & \ell \in [\infty,0]
\end{align*}
and
\begin{align*}
\Gamma_{R,4}(s) &= \begin{pmatrix} 1 & 0 \\ 0 & 1 \end{pmatrix}  +\frac12 \frac{2}{1+2is} P_s \begin{pmatrix} 1 & 1 \\ 1 & 1 \end{pmatrix}, \quad &s \in [\infty,-\infty], \\
\Gamma_{R,j} &= \begin{pmatrix} 1 & 0 \\ 0 & 1 \end{pmatrix}, \quad & j \in \{1,2,3,5,6\}.
\end{align*}
Explicit computation then shows
\begin{align*}
\Gamma_{S,j}^* \left( \Gamma_{R,j} - \begin{pmatrix} 1 & 0 \\ 0 & 1 \end{pmatrix} \right) &= \Gamma_{R,j} - \begin{pmatrix} 1 & 0 \\ 0 & 1 \end{pmatrix}.
\end{align*}
In particular we find 
\begin{align*}
\Gamma_{S,j}^* \left( \Gamma_{R,j} - \begin{pmatrix} 1 & 0 \\ 0 & 1 \end{pmatrix} \right) &= \begin{pmatrix} 0 & 0 \\ 0 & 0 \end{pmatrix}
\end{align*}
if $j \in \{1,2,3,5,6\}$ and
\begin{align*}
\Gamma_{S,4}^* (s)\left( \Gamma_{R,4}(s) - \begin{pmatrix} 1 & 0 \\ 0 & 1 \end{pmatrix} \right) &= -\frac12 \frac{2}{1+2is} P_s \begin{pmatrix} 1 & 1 \\ 1 & 1 \end{pmatrix}
\end{align*}
for $s \in [\infty,-\infty]$. Thus we find $q(\mathcal{V}\mathcal{U}F_0 (W_S^*(W_R-\textup{Id})) F_0^* \mathcal{U}^*\mathcal{V}^*) = q(\mathcal{V}\mathcal{U}F_0 (W_R-\textup{Id})F_0^* \mathcal{U}^*\mathcal{V}^*)$. Since their images under the quotient map agree, we have $W_R = \textup{Id}+W_S^*(W_R-\textup{Id})$ (mod compacts). The result then follows from Equation \eqref{eq:factorisation}.
\end{proof}

\begin{lemma}
We have the equality
\begin{align*}
\textup{Index}(W_R) &=  \dim(P_s).
\end{align*}
\end{lemma}
\begin{proof}
The operator $q(\mathcal{V}\mathcal{U} F_0 W_R F_0^* \mathcal{U}^* \mathcal{V}^*)$ has components \begin{align*}
\Gamma_{R,4}(s) &= \begin{pmatrix} 1 & 0 \\ 0 & 1 \end{pmatrix} - \frac12 \frac{2}{1+2is} P_s \begin{pmatrix} 1 & 1 \\ 1 & 1 \end{pmatrix}, \quad &s \in [\infty,-\infty], \\
\Gamma_{R,j} &= \begin{pmatrix} 1 & 0 \\ 0 & 1 \end{pmatrix}, \quad & j \in \{1,2,3,5,6\}.
\end{align*}
The index of $W_R$ can then be computed using Gohberg-Kre\u{\i}n  theory \cite{GK}  as the sum of winding numbers 
\begin{align*}
\textup{Index}(W_R) &= \sum_{j=1}^6{\textup{Wind}(\Gamma_{R,j})}.
\end{align*}
For $j \in \{1,2,3,5,6\}$ we have $\textup{Wind}(\Gamma_{R,j}) = 0$. If $\dim(P_s) = 0$ we have $\textup{Wind}(\Gamma_{R,4}) = 0$ also. If $\dim(P_s) = 1$, then we find
\begin{align*}
\textup{Wind}(\Gamma_{R,4}) &= - \frac{1}{2\pi i} \int_{-\infty}^\infty{ \frac{ \frac{\d }{\d s} \det(\Gamma_{R,4}(s))}{\det(\Gamma_{R,4}(s))} \, \d s} \\
&= \frac{1}{2\pi i} \int_{-\infty}^\infty{ \frac{4i}{4s^2+1} \, \d s} \\
&= 1,
\end{align*}
from which the result follows.
\end{proof}

\begin{cor}\label{cor:index-resonance}
We have the equality
\begin{align*}
\textup{Index}(W_-) &= \textup{Index}(W_S) + \dim(P_s).
\end{align*}
\end{cor}

To determine an analytic formula for $\textup{Index}(W_S)$ requires a more subtle analysis of the high-energy behaviour of the scattering matrix (see for example \cite[Theorem III.1]{guillope81} and \cite[Section 9.2]{yafaev10}) and we defer the proof to a future publication. 
The analytic formula has previously been obtained by \cite{jia12}, and reads
\begin{align}\label{eq:jia-levinson-again}
\textup{Index}(W_-)= -N=\frac{1}{2\pi i} \int_0^\infty \left( \textup{Tr} \left(S(\lambda)^*S'(\lambda) \right) - c_1 \right) \, \d \lambda + \dim(P_s) - \beta_2. 
\end{align}
Here the constants are given by
\[
c_1= -\frac{(2\pi i) \textup{Vol}(\Sf^3)}{2(2\pi)^4} \int_{\R^4}V(x)\,\d x\quad\mbox{and}\quad \beta_2=- \frac{\textup{Vol}(\Sf^3)}{4(2\pi)^4} \int_{\R^4}V(x)^2\, \d x.
\]
The proof in \cite{jia12} used high energy asymptotics of the spectral shift function, obtained using heat kernel methods. Our approach also uses high energy asymptotics for the spectral shift, but then utilises regularised determinants, the limiting absorption principle and Gohberg-Kre\u{\i}n theory.

\end{document}